\documentclass[a4paper,12pt,reqno]{amsart}
\usepackage{amsmath,amsfonts,amssymb,amsthm,enumerate,multicol}
\usepackage{tikz,graphicx}
\usepackage{hyperref}
\usepackage{caption,enumitem}
\usepackage{cleveref}
\usepackage{float}
\usepackage[labelsep=space]{caption}
\usepackage[font=footnotesize]{caption}
\usepackage{xcolor}
\usepackage{cite}
\usepackage{mathtools}
\usepackage{autobreak}

\definecolor{eqblue}{RGB}{0,0,180} \definecolor{citered}{RGB}{180,0,0} \hypersetup{ colorlinks=true, linkcolor=eqblue, citecolor=citered,  urlcolor=blue }
\theoremstyle{plain}
\newtheorem{theorem}{Theorem}[section]
\newtheorem{prop}[theorem]{Proposition}
\newtheorem{corollary}[theorem]{Corollary}
\newtheorem{lemma}[theorem]{Lemma}
\theoremstyle{definition}
\newtheorem{definition}[theorem]{Definition}
\theoremstyle{remark}
\newtheorem{remark}[theorem]{Remark}

\allowdisplaybreaks
\numberwithin{equation}{section}

\begin{document}
\thispagestyle{empty}

\begin{center}
\Large{\textbf{On a Class of Continuous Collision-Induced Breakage Equation}}
\end{center}
\medskip
\centerline{${\text{Mashkoor~Ali}}$}
\let\thefootnote\relax
\footnotetext{Email address: mashkoor.ali@jgu.edu.in}
\medskip
{\footnotesize
  \centerline{Jindal Global Business School, O.P. Jindal Global University,}
  \centerline{Sonipat-131001, Haryana, India}
}
\bigskip

\begin{quote}
{\small{\em\bf Abstract.}
In this work, we establish the existence of mass-conserving weak solutions
to a nonlinear collision-induced breakage equation in which binary
collisions may trigger particle breakup. The result is proved for a class
of product-type collision kernels whose small-size behavior is controlled by
a power-law function of the form $\omega_0(x)\le A_1\,x^\ell$,
while no growth restriction is imposed on the large-size factor
$\omega_\infty$.
The qualitative behavior of the solutions depends crucially on the exponent
$\ell$ near the origin. Sublinear growth corresponding to $\ell<\tfrac12$
yields existence only on finite time intervals, whereas superlinear growth
corresponding to $\ell>\tfrac12$ ensures global-in-time existence.}
\end{quote}

\vspace{.3cm}


\section{\textbf{Introduction}}
%%%%%%%%%%%%%%%%%%%%%%%%%%%%%%%%%%%%

Binary collision processes leading to the breakage of clusters play an
important role in several physical and astrophysical phenomena.
The corresponding evolution of the cluster-size distribution is commonly
described by the collision-induced, or nonlinear, breakage equation.
Such models arise naturally in the study of cloud droplet dynamics and in
theories related to the formation of planetary bodies \cite{LG1976,SV1972,SRC1978,BKB2015}. The evolution of the cluster density is governed by an integro-differential equation balancing the creation and depletion of clusters through
collision-induced breakage events. More precisely, we study
\begin{subequations}
\begin{align}
    \frac{\partial}{\partial t}f(t,x)
    &= \int_0^{\infty}\!\int_x^{\infty}
       b(x,y,z)\,a(y,z)\,f(t,y)\,f(t,z)\,dy\,dz
       -\int_0^{\infty} a(x,y)\,f(t,x)\,f(t,y)\,dy,
    \label{NLBE}\\
    f(0,x) &= f^{\rm in}(x).
    \label{IC}
\end{align}
\end{subequations}
In \eqref{NLBE}, the collision kernel $a$ determines the rate at which
particles collide, while the daughter distribution function $b$ describes
the fragments produced after a collision event. In the no-mass-transfer
setting, we assume that $b$ is a nonnegative function such that, for
$(x,y,z)\in(0,\infty)^3$, the quantity $b(x,y,z)$ denotes the
distribution of fragments of size $x\in(0,y)$ generated from a particle
of size $y$ after colliding with a particle of size $z$. Moreover,
\begin{align}
\int_0^y x\,b(x,y,z)\,dx = y,
\qquad
b(x,y,z)=0 \quad \text{for } x>y.
\label{LMC}
\end{align}
The above class of daughter distributions excludes the transfer of mass
between colliding particles during breakup, since fragments generated from
a particle of size $y$ arise only from that particle and carry exactly its
mass. As a consequence, no mass is exchanged between colliding partners
during fragmentation. In equation~\eqref{NLBE}, the gain term describes
the creation of particles of size $x$ through collisions between particles
of sizes $y$ and $z$, with $y>x$ and $z\in(0,\infty)$, whose fragmentation
may produce daughter particles of size $x$. The loss term accounts for the
removal of particles of size $x$ as they undergo collisions with particles
of arbitrary size. Since the focus of the present work is on nonlinear
breakage processes, where collisions between two clusters lead to the breakage
of one of the colliding clusters into smaller clusters, we also mention the
class of linear breakage equations for comparison. In linear breakage models,
clusters undergo breakage due to internal mechanisms such as instability,
decay, or other intrinsic effects, rather than through binary collisions.
The mathematical theory of linear breakage equations with and without coagulation is well developed and has been extensively studied from both mathematical and physical perspectives.
We refer to the monograph \cite{BLL2019} for a comprehensive overview of
the subject and further references therein.

Several works in the physics literature have focused on the qualitative and
asymptotic behaviour of collision-induced breakage processes. Different
fragmentation scenarios, including equal binary splitting and breakage of one
of the colliding particles, were analysed in \cite{CHNG1988,CHNG1990}.
Travelling-wave techniques were later employed in \cite{Krapivsky2003} to
describe the evolution of the fragment size distribution. Approximate
analytical descriptions based on Gamma distributions were also proposed in
\cite{KK2000,KK2006}. In addition, for the product-type kernel
$a(x,y)=(xy)^{\lambda/2}$ with $0\leq \lambda\leq2$, the asymptotic dynamics
were studied in \cite{EP2007}, where the nonlinear model was shown to admit a
connection with a corresponding linear breakage equation. The mathematical study of the nonlinear collision-induced breakage equation without mass transfer was initiated in \cite{AKG2021I}, where existence and uniqueness of weak solutions were investigated for collision kernels of the form
\begin{align*}
a(x,y)=x^{\alpha}y^{\beta}+x^{\beta}y^{\alpha},
\qquad \alpha\leq \beta \leq 1.
\end{align*}
The analysis revealed a threshold behaviour with respect to the homogeneity
$\alpha+\beta$: global mass-conserving weak solutions were obtained for
$\alpha+\beta\in[1,2]$, while only local-in-time mass-conserving solutions
exist for $\alpha+\beta\in[0,1)$. The theory was later extended in
\cite{RJ2024} to include non-integrable daughter distribution functions,
where analogous existence, uniqueness, and non-existence results were
established under suitable assumptions on the singularity of the fragmentation
kernel. More recently, the existence of mass-conserving self-similar solutions,
together with qualitative properties of the scaling profiles, was studied in
\cite{RJ2025} using compactness and dynamical methods. For the discrete setting, the discrete counterpart of \eqref{NLBE}--\eqref{IC} has been investigated in\cite{AliLaurencot2026}, where global existence of solutions is established for a broad class of collision kernels without imposing any growth assumptions. Several works addressing existence and uniqueness in which coagulation is incorporated together with collision-induced breakage are also available; see\cite{Barik2020,BarikGiri2020,BarikGiri2021,GiriLaurencot2021}. The present work is motivated by the study of Lauren\c{c}ot \cite{Laurencot2000}, where the existence of solutions to the continuous coagulation--fragmentation equation was established for
product-type coagulation kernels under weak fragmentation assumptions.

\medskip
\noindent\textbf{Kernel class.}
In this article, we discuss the well-posedness of \eqref{NLBE}--\eqref{IC}
for the following class of collision kernels. We assume
\begin{align}
a(x,y)=a(y,x)\ge0,
\qquad (x,y)\in \mathbb{R}_+^2,
\label{eq:kernel-asm1}
\end{align}
and, more specifically,
\begin{align}
a(x,y)= A_0
\begin{cases}
\omega_0(x)\,\omega_0(y), & 0<x,y\le 1,\\[1mm]
\omega_0(x)\,\omega_\infty(y), & x\le 1<y,\\[1mm]
\omega_\infty(x)\,\omega_0(y), & y\le 1<x,\\[1mm]
\omega_\infty(x)\,\omega_\infty(y), & x,y\ge 1,
\end{cases}
\label{eq:kernel-class}
\end{align}
where $\omega_0,\omega_\infty:[0,\infty)\to[0,\infty)$ are continuous
non-decreasing functions and $A_0>0$ is a positive constant.
In other words, $a(x,y)=A_0\,\omega(x)\,\omega(y)$, where
$\omega=\omega_0$ on $(0,1]$ and $\omega=\omega_\infty$ on $(1,\infty)$.
A symmetric kernel belongs to this class if and only if it can be written in
the factored form $\psi(x)\psi(y)$ for some non-decreasing $\psi\ge0$.

The following kernels belong to the class \eqref{eq:kernel-class} and
illustrate a variety of admissible growth behaviours near the origin and
at infinity.
\begin{itemize}
\item[(I)]
$a(x,y)=A_0\,(xy)^\ell$,\quad $\ell>0$,\quad
$\omega(x)=x^\ell$.
\item[(II)]
$a(x,y)=A_0\,(xy)^\ell(1+x)^\beta(1+y)^\beta$,\quad
$\ell>0,\ \beta\ge0$,\quad $\omega(x)=x^\ell(1+x)^\beta$.
\item[(III)]
$a(x,y)=A_0\,(xy)^\ell e^{\gamma(x+y)}$,\quad
$\ell>0,\ \gamma>0$,\quad $\omega(x)=x^\ell e^{\gamma x}$.
\item[(IV)]
$a(x,y)=A_0\,(xy)^\ell
(\log(1+x))^\gamma(\log(1+y))^\gamma$,\quad
$\ell>1,\ \gamma>0$,\quad
$\omega(x)=x^\ell(\log(1+x))^\gamma$.
\item[(V)]
$a(x,y)=A_0\,(xy)^\ell e^{\gamma(x^\nu+y^\nu)}$,\quad
$\ell>0,\ \gamma>0,\ \nu>0$,\quad
$\omega(x)=x^\ell e^{\gamma x^\nu}$.
\item[(VI)]
$a(x,y)=A_0\,\dfrac{(xy)^\ell}{(1+x)^\mu(1+y)^\mu}$,\quad
$\ell>\mu>0$,\quad
$\omega(x)=\dfrac{x^\ell}{(1+x)^\mu}$.
\item[(VII)]
$a(x,y)=A_0\,(xy)^\ell(2-e^{-x})(2-e^{-y})$,\quad
$\ell>0$,\quad $\omega(x)=x^\ell(2-e^{-x})$.
\item[(VIII)]
$a(x,y)
=A_0\,\omega_0(x)\omega_0(y)\,\mathbf{1}_{\{x,y\le1\}}
+A_0\,x^p y^p\,\mathbf{1}_{\{x,y\ge1\}}$,\quad
$\omega_0(x)=x^\ell$,\quad $p>\ell>0$.
\end{itemize}

In addition to the conservation of mass \eqref{LMC}, we assume that
the number of fragments produced in a collisional breakage event is
uniformly bounded independently of the sizes of the colliding particles.
More precisely, there exists $\beta_0\ge2$ such that
\begin{align}
\int_0^y b(x,y,z)\,dx \le \beta_0,
\qquad (y,z)\in(0,\infty)^2.
\label{NOP}
\end{align}
In addition to \eqref{LMC} and \eqref{NOP}, we impose the following
structural assumption on the daughter distribution $b$. There exist
$p\in(1,2)$ and $B_p>0$ such that
\begin{align}
\int_0^y b(x,y,z)^p\,dx
\le \frac{B_p}{2}\,y^{1-p},
\qquad (y,z)\in(0,\infty)^2.
\label{eq:p-cond}
\end{align}

The main result of this paper is the following theorem.

\begin{theorem}\label{Thm:GE}
Assume \eqref{eq:kernel-asm1}--\eqref{eq:kernel-class},
\eqref{LMC}, \eqref{NOP}, \eqref{eq:p-cond},
and \eqref{eq:w0-growth}. Let $f^{\mathrm{in}}
\in \Xi_0 \cap \Xi_g^+$. Then the following hold.
\begin{enumerate}
\item[\upshape(i)]
\textbf{Global existence.}
Assume condition \eqref{eq:Case2}, i.e., $\ell > \tfrac{1}{2}$.
Then there exists a global weak solution
\[
f \in \mathcal{C}([0,\infty);\Xi_0)
\cap L^\infty((0,\infty);\Xi_g^+)
\]
to \eqref{NLBE}--\eqref{IC} in the sense of Definition~\ref{DEF1},
satisfying
\begin{align}
\int_0^\infty x\,f(t,x)\,dx \le \Theta,
\qquad t \ge 0.
\label{eq:mass-ineq-limit}
\end{align}
\item[\upshape(ii)]
\textbf{Finite-time existence.}
Assume condition \eqref{eq:Case1}, i.e.,
$\ell \in \bigl(0,\tfrac{1}{2}\bigr)$.
Then for every $T > 0$ there exists a weak solution
\[
f \in \mathcal{C}([0,T];\Xi_0) \cap L^\infty((0,T);\Xi_g^+)
\]
to \eqref{NLBE}--\eqref{IC} on $[0,T]$
in the sense of Definition~\ref{DEF1},
satisfying
\[
\int_0^\infty x\,f(t,x)\,dx \le \Theta,
\qquad t \in [0,T].
\]
\end{enumerate}
In both cases the solution is obtained as the limit of the truncated
solutions $(f_n)_{n\ge1}$ constructed in Proposition~\ref{prop:tr-exist},
and $\Theta := \int_0^\infty x\,f^{\mathrm{in}}(x)\,dx$ denotes the
initial first moment.
\end{theorem}

The remainder of the paper is organized as follows.
In Section~\ref{sec:framework} we introduce the functional framework,
including the weighted Lebesgue spaces, the admissible weight class
$\mathcal{G}_b$, the precise notion of weak solution to
\eqref{NLBE}--\eqref{IC}, and the truncated approximating problem.
Section~\ref{sec:estimates} is devoted to the derivation of uniform
estimates for the truncated solutions: we establish existence of the
truncated solutions in Proposition~\ref{prop:tr-exist}, derive weighted
moment bounds in Corollary~\ref{COR1}, obtain tail and large-size
estimates in Lemma~\ref{Lem:Tail}, control the zeroth moment in
Lemma~\ref{Lem:ZM}, and prove uniform integrability and time
equicontinuity in Lemmas~\ref{Lem:UI} and~\ref{Lem:time-eqc}
respectively.
The proof of Theorem~\ref{Thm:GE} is carried out in
Section~\ref{sec:proof}, where the compactness argument is completed
via the Dunford--Pettis theorem and the Arzel\`{a}--Ascoli theorem, and
the limit is identified as a weak solution by passing to the limit in
the truncated weak formulation.
Finally, Section~\ref{sec:mass} establishes the mass-conservation
property of the constructed solution in Proposition~\ref{Prop:MC}.
%%%%%%%%%%%%%%%%%%%%%%%%%%%%%%%%%%%%
\section{\textbf{Functional Framework}}\label{sec:framework}
%%%%%%%%%%%%%%%%%%%%%%%%%%%%%%%%%%%%

\subsection*{Function spaces}

Let $g:(0,\infty)\to[0,\infty)$ be a measurable weight function. We define
the weighted Lebesgue space
\[
\Xi_g:=L^1\bigl((0,\infty),g(y)\,dy\bigr),
\]
equipped with the norm
\[
\|H\|_{\Xi_g}
:=
\int_0^\infty |H(y)|\,g(y)\,dy,
\qquad H\in \Xi_g,
\]
and the associated weighted moment functional
\[
M_g(H)
:=
\int_0^\infty H(y)\,g(y)\,dy,
\qquad H\in \Xi_g.
\]
The positive cone of $\Xi_g$ is
\[
\Xi_g^{+}:=\{H\in \Xi_g \mid H(y)\ge0 \ \text{a.e.\ on } (0,\infty)\},
\]
and $\Xi_{g,w}$ denotes $\Xi_g$ endowed with its weak topology.

An important family of weights is given by
$g_m(y):=y^m$, $y>0$, $m\in\mathbb{R}$.
In this case we write $\Xi_m:=\Xi_{g_m}$ and define
\[
M_m(q)
:=
\int_0^\infty y^m q(y)\,dy,
\qquad q\in \Xi_m.
\]
In particular, $\Xi_0=L^1(0,\infty)$,
while $M_0(h)$ represents the total number of particles and
$M_1(h)$ the total mass in the system.

\subsection*{Admissible weight class}

We consider weight functions $g$ belonging to the admissible class
\begin{align}
\mathcal{G}_b :=
\Bigg\{
g \in L^1_{\mathrm{loc}}(\mathbb{R}_+) \;\Big|\;
g(x)>0 \ \text{for all } x>0,\quad
\frac{g(x)}{x}\ \text{non-decreasing on }(0,\infty),\notag\\
g(y)-\int_0^y g(x)\,b(x,y,z)\,dx
\ge \theta\,g(y),
\quad (y,z)\in(0,\infty)^2
\Bigg\},
\label{def:G}
\end{align}
for some constant $\theta \in (0,1)$ depending only on $g$. Throughout,
we set
\[
C_0:=\int_0^{\infty}g(x)\,f^{\rm in}(x)\,dx<\infty.
\]

\begin{remark}\label{rem:general-admissible}
For the daughter distribution
\[
b(x,y,z)=\frac{\nu+2}{y^{\nu+1}}x^\nu,
\qquad \nu\in(-1,0],
\]
the dissipativity condition
\[
g(y)-\int_0^y g(x)\,b(x,y,z)\,dx
\ge \theta_\alpha\,g(y),
\qquad \theta_\alpha=\frac{\alpha-1}{\nu+\alpha+1}\in(0,1),
\]
is satisfied by several natural classes of weight functions.
In particular, the following families belong to the admissible class
$\mathcal{G}_b$ whenever $\alpha>1$,
\begin{align*}
g(x)&=x^{\alpha},\\
g(x)&=x^\alpha(1+x)^\beta,
\qquad \beta\ge0,\\
g(x)&=x^\alpha e^{\lambda x},
\qquad \lambda>0,\\
g(x)&=x^\alpha(\log(1+x))^\gamma,
\qquad \gamma>0.
\end{align*}
More generally, any function of the form $g(x)=x^\alpha h(x)$ with
$\alpha>1$ and $h:(0,\infty)\to[0,\infty)$ non-decreasing satisfies
the dissipativity condition with constant $\theta_\alpha$.
\end{remark}

Other examples of daughter distribution functions, originally introduced
in \cite[Section~6]{KLL2025} in the context of the linear fragmentation
equation, are given by
\[
b(x,y,z)=
\begin{cases}
1,
& x\in [0,1]\cup [y-1,y],\quad y>2,\\[2mm]
0,
& x\in (1,y-1),\quad y>2,\\[2mm]
\dfrac{2}{y},
& y\le 2,
\end{cases}
\]
and
\[
b(x,y,z)=
\begin{cases}
y,
& x\in \bigl[0,\tfrac1y\bigr]
\cup
\bigl[y-\tfrac1y,y\bigr],\quad y>\sqrt2,\\[3mm]
0,
& x\in \bigl(\tfrac1y,\,y-\tfrac1y\bigr),\quad y>\sqrt2,\\[3mm]
\dfrac{2}{y},
& y\le \sqrt2.
\end{cases}
\]

We now state the definition of a weak solution to
\eqref{NLBE}--\eqref{IC} used throughout the paper.

\begin{definition}\label{DEF1}
Let $T \in (0,\infty]$ and let the daughter distribution $b$ satisfy
\eqref{LMC} and \eqref{NOP}. Given $f^{\rm in} \in \Xi_0 \cap \Xi_g^+$,
a \emph{weak solution} to \eqref{NLBE}--\eqref{IC} on $[0,T)$ is a
nonnegative function
\[
f \in \mathcal{C}([0,T);\Xi_{0,w}) \cap L^\infty((0,T);\Xi_g^+)
\]
such that $(s,x,y) \mapsto a(x,y)f(s,x)f(s,y)
\in L^1\bigl((0,t)\times(0,\infty)^2\bigr)$,
and, for all $t \in (0,T)$ and $\psi \in L^\infty(0,\infty)$,
\begin{align}
\int_0^\infty \psi(x)\bigl(f(t,x)-f^{\rm in}(x)\bigr)\,dx
= \frac{1}{2}\int_0^t \int_0^\infty \int_0^\infty
\tilde{\psi}(y,z)\, a(y,z)\, f(s,y)\,f(s,z)\,dy\,dz\,ds,
\label{eq:def-wf}
\end{align}
where
\begin{align}
\tilde{\psi}(y,z)
:= \int_0^{y} \psi(x)\, b(x,y,z)\,dx - \psi(y),
\qquad (y,z)\in (0,\infty)^2.
\label{def:tilde-psi}
\end{align}
Moreover, a weak solution $f$ is said to be \emph{mass-conserving} if
$M_1(f(t)) = M_1(f^{\rm in})$ for all $t \in [0,T)$.
\end{definition}

\subsection*{Truncated problem}

Next, we define the approximation used to construct solutions.
Set
\begin{align}
f_n^{\rm in}(x) = f^{\rm in}(x)\,\mathbf{1}_{(0,n)}(x),
\qquad
a_n(y,z) = a(y,z)\,\mathbf{1}_{(0,n)}(y)\,\mathbf{1}_{(0,n)}(z),
\end{align}
and consider the truncated integro-differential equation
\begin{subequations}
\begin{align}
\frac{\partial}{\partial t}f_n(t,x)
&= \int_0^{n}\!\int_x^{n}
   b(x,y,z)\,a(y,z)\,f_n(t,y)\,f_n(t,z)\,dy\,dz
   -\int_0^{n} a(x,y)\,f_n(t,x)\,f_n(t,y)\,dy,
\label{Tr:NLBE}\\
f_n(0,x) &= f_n^{\rm in}(x).
\label{Tr:IC}
\end{align}
\end{subequations}
For a weight function $g$, the truncated weighted moment of the approximate solution $f_n$ is
\begin{align*}
M_g^n(t) :=\int_0^n g(x)\,f_n(t,x)\,dx, \qquad t\ge0.
\end{align*}
In the special case $g(x)=x^m$, $m\ge0$, we write
\[
M_m^n(t)
:=
\int_0^n x^m\,f_n(t,x)\,dx.
\]
The quantity $M_0^n(t)=\int_0^n f_n(t,x)\,dx$ is the zeroth moment (total
number of particles), and $M_1^n(t)=\int_0^n x\,f_n(t,x)\,dx$ is the
first moment (total mass) in the truncated system.

%%%%%%%%%%%%%%%%%%%%%%%%%%%%%%%%%%%%
\section{\textbf{Uniform Estimates}}\label{sec:estimates}
%%%%%%%%%%%%%%%%%%%%%%%%%%%%%%%%%%%%
The constants $A_0>0$ and $A_1>0$ introduced in \eqref{eq:kernel-class} and \eqref{eq:w0-growth} are fixed throughout. In all proofs and intermediate estimates, any finite combination of $A_0$ and $A_1$ is absorbed into a single generic positive constant $A$, independent of $n$, whose value may change from line to line.
%%%%%%%%%%%%%%%%%%%%%%%%%%%%%%%%%%%%%%%%%%%%%%%%%%%%%%%%%%%%%%%%%%%%%%%%%%%%%%%%%%%%%%%%%%%%%%%%%%%%%%%%%%%%%%%%%%%%%%%%%%%%%%%%%%%%%%%%%%%%%%%%%%%%%%%%%%%%%%%%%%%%%%%%%%%%%%%%%%%%%%%%%%%%%%%%%%%%%%%%%%%%%%%%%%%%%%%%%%%%%%%%%%%%%%%%%%%%%%%%%%%%%%%%%%%%%%%%%%%%%%%%%%%%%%%%%%%%
\begin{prop}\label{prop:tr-exist}
Let $n \ge 1$. Then there exists a unique non-negative solution
$f_n \in C^1\bigl([0,\infty);\,L^1(0,n)\bigr)$
to \eqref{Tr:NLBE}--\eqref{Tr:IC}. Moreover, for $g\in L^{\infty}(0,n)$,
the following identity holds for all $t \ge 0$,
\begin{subequations}
\begin{align}
\int_0^n g(x)\,f_n(t,x)\,dx
&+
\int_0^t\!\int_0^n\!\int_0^n \tilde{g}(y,z)\,
a(y,z)\,f_n(s,y)\,f_n(s,z)\,dy\,dz\,ds= \int_0^n g(x)\,f_n^{\mathrm{in}}(x)\,dx 
\label{Tr:WF}
\end{align}

In addition, $f_n$ satisfies the truncated mass conservation property
\begin{align}
\int_0^n x\,f_n(t,x)\,dx
=
\int_0^n x\,f_n^{\mathrm{in}}(x)\,dx,
\qquad t\ge0.
\label{eq:TrMC}
\end{align}
\end{subequations}
\end{prop}

\begin{proof}
The proof follows along the same lines as
\cite[Proposition~3.1]{AKG2021I}, and is therefore omitted.
\end{proof}

In particular, the truncated mass conservation \eqref{eq:TrMC} together
with the monotone truncation of the initial data gives
\begin{align}
\int_0^n x\,f_n(t,x)\,dx
= \int_0^n x\,f_n^{\mathrm{in}}(x)\,dx
\le \int_0^\infty x\,f^{\mathrm{in}}(x)\,dx
= \Theta,
\qquad t \ge 0.
\label{eq:mass-ineq}
\end{align}

\begin{corollary}\label{COR1}
Let $g \in \mathcal{G}_b$ and let $f_n$ be the solution to
\eqref{Tr:NLBE}--\eqref{Tr:IC} given by Proposition~\ref{prop:tr-exist}.
Then, for all $t \ge 0$,
\begin{subequations}
\begin{align}
\int_0^n g(x)\,f_n(t,x)\,dx
&\le \frac{C_0}{\theta},
\label{eq:HMBD1}\\[4pt]
\int_0^t \int_0^n \int_0^n
g(y)\,a(y,z)\,f_n(s,y)\,f_n(s,z)\,dy\,dz\,ds
&\le \frac{C_0}{\theta}.
\label{eq:HMBD2}
\end{align}
\end{subequations}
\end{corollary}

\begin{proof}
Since $g\in\mathcal{G}_b$, we have $g\in L^\infty(0,n)$ for every $n>0$,
and therefore $g$ is an admissible test function in \eqref{Tr:WF}.
Using the dissipativity condition \eqref{def:G} together with the
non-negativity of $f_n$, we deduce from \eqref{Tr:WF} that
\begin{align}
\int_0^n g(x)\,f_n(t,x)\,dx
\;+\; \theta
\int_0^t \int_0^n \int_0^n
g(y)\,a(y,z)\,f_n(s,y)\,f_n(s,z)\,dy\,dz\,ds
&\notag\\
\;\le\;
\int_0^n g(x)\,f_n^{\mathrm{in}}(x)\,dx
\;\le\; C_0.&
\label{eq:cor-combined}
\end{align}
Since both terms on the left-hand side of \eqref{eq:cor-combined} are
non-negative, dropping each in turn and dividing by $\theta \in (0,1)$
yields \eqref{eq:HMBD1} and \eqref{eq:HMBD2}.
\end{proof}

\begin{lemma}\label{Lem:Tail}
Let $m\in(1,n)$ and $t>0$. Assume that the hypotheses of
Corollary~\ref{COR1} are satisfied. Then
\begin{subequations}
\begin{align}
\int_m^n g(x)\,f_n(t,x)\,dx
&\le \int_m^n g(x)\,f_n^{\mathrm{in}}(x)\,dx,
\label{eq:tail1}\\[4pt]
\theta
\int_0^t \int_0^n \int_m^n
g(y)\,a(y,z)\,f_n(s,y)\,f_n(s,z)\,dy\,dz\,ds
&\le \int_m^n g(x)\,f_n^{\mathrm{in}}(x)\,dx,
\label{eq:tail2}\\[4pt]
\int_0^t
\left(
\int_m^n \omega_\infty(y)\,f_n(s,y)\,dy
\right)^{\!2}ds
&\le \frac{1}{\theta\,g(m)}
\int_m^n g(x)\,f_n^{\mathrm{in}}(x)\,dx.
\label{eq:tail3}
\end{align}
\end{subequations}
\end{lemma}

\begin{proof}
Define the truncated function
$\hat{g}(x) := g(x)\,\mathbf{1}_{(m,n)}(x)$, $x > 0$.
Since $g$ is non-negative and locally bounded, so is $\hat{g}$,
and hence $\hat{g}$ is an admissible test function in \eqref{Tr:WF}.
We note, however, that $\hat{g} \notin \mathcal{G}_b$ in general, since
the global dissipativity condition in \eqref{def:G} need not be satisfied
by $\hat{g}$; consequently, Corollary~\ref{COR1} cannot be applied
directly with $\hat{g}$ in place of $g$.

Substituting $\hat{g}$ into \eqref{Tr:WF} and using
$\hat{g}(y) = 0$ for $y \notin (m,n)$, we obtain
\begin{align}
\int_m^n g(x)\,f_n(t,x)\,dx
+
\int_0^t\!\int_0^n\!\int_m^n
\!\left(
g(y) - \int_0^y \hat{g}(x)\,b(x,y,z)\,dx
\right)
a(y,z)\,f_n(s,y)\,f_n(s,z)\,dy\,dz\,ds
&\notag\\
= \int_m^n g(x)\,f_n^{\mathrm{in}}(x)\,dx.&
\label{eq:WF-tg}
\end{align}
For $y \in (m,n)$, since $\hat{g}(x) = 0$ for $x \le m$,
the inner integral in \eqref{eq:WF-tg} simplifies to
$\int_0^y \hat{g}(x)\,b(x,y,z)\,dx = \int_m^y g(x)\,b(x,y,z)\,dx$,
so that \eqref{eq:WF-tg} becomes
\begin{align}
\int_m^n g(x)\,f_n(t,x)\,dx
+\int_0^t\!\int_0^n\!\int_m^n
\Bigl(g(y) - \int_m^y g(x)\,b(x,y,z)\,dx\Bigr)
a(y,z)\,f_n(s,y)\,f_n(s,z)\,dy\,dz\,ds
&\notag\\
= \int_m^n g(x)\,f_n^{\mathrm{in}}(x)\,dx.&
\label{eq:WF-tg2}
\end{align}
We now estimate the dissipation term in \eqref{eq:WF-tg2} from below
using the dissipativity property of $g$.
For any $y \in (m,n)$ and $z \in (0,\infty)$, since $g \ge 0$
and $b \ge 0$, restricting the domain of integration gives
\[
\int_m^y g(x)\,b(x,y,z)\,dx
\le \int_0^y g(x)\,b(x,y,z)\,dx,
\]
and therefore
\[
g(y) - \int_m^y g(x)\,b(x,y,z)\,dx
\ge g(y) - \int_0^y g(x)\,b(x,y,z)\,dx
\ge \theta\,g(y),
\]
where the last inequality is the defining property of $\mathcal{G}_b$
applied to $g$. Substituting into \eqref{eq:WF-tg2} and using
$f_n \ge 0$, we arrive at
\begin{align}
\int_m^n g(x)\,f_n(t,x)\,dx
\;+\; \theta
\int_0^t\!\int_0^n\!\int_m^n
g(y)\,a(y,z)\,f_n(s,y)\,f_n(s,z)\,dy\,dz\,ds
\;\le\;
\int_m^n g(x)\,f_n^{\mathrm{in}}(x)\,dx.
\label{eq:tail-combined}
\end{align}
Since both terms on the left of \eqref{eq:tail-combined} are non-negative,
dropping each one in turn immediately yields \eqref{eq:tail1} and
\eqref{eq:tail2}.

It remains to establish \eqref{eq:tail3}. Applying the kernel structure
\eqref{eq:kernel-class}, namely $a(y,z)=A_0\,\omega_\infty(y)\omega_\infty(z)$
for $y,z \ge m \ge 1$, and using $g(y) \ge g(m) > 0$ for $y \in (m,n)$
together with \eqref{eq:tail2}, we obtain
\begin{align}
\int_m^n g(x)\,f_n^{\mathrm{in}}(x)\,dx
&\ge
\theta\,g(m)\,A_0
\int_0^t\!\int_0^n\!\int_m^n
\omega_\infty(y)\,\omega_\infty(z)\,
f_n(s,y)\,f_n(s,z)\,dy\,dz\,ds
\notag\\[4pt]
&\ge
\theta\,g(m)\,A_0
\int_0^t
\left(
\int_m^n \omega_\infty(y)\,f_n(s,y)\,dy
\right)^{\!2}ds,
\label{eq:tail3-chain}
\end{align}
where the last inequality follows from the fact that the integral of
$\omega_\infty(z)\,f_n(s,z)$ over $(0,n)$ dominates that over $(m,n)$.
Dividing both sides of \eqref{eq:tail3-chain} by $\theta\,g(m)\,A_0 > 0$
(absorbing $A_0$ into the right-hand side constant)
yields \eqref{eq:tail3}.
\end{proof}

\medskip

The behaviour of the zeroth moment is governed by the growth of $\omega_0$
near the origin. We assume throughout that
\begin{align}
\omega_0(x) \le A_1\,x^{\ell},
\qquad x \in (0,1),\quad \ell \ge 0,
\label{eq:w0-growth}
\end{align}
for some constant $A_1 > 0$, and distinguish the following two
regimes according to whether $2\ell$ is less than or greater than $1$.
\begin{itemize}
\item[\textbf{Case 1.}]
$\ell \in \bigl(0,\tfrac{1}{2}\bigr)$,
\hfill\refstepcounter{equation}
(\theequation)\label{eq:Case1}

\noindent
so that $2\ell < 1$. In this sub-linear regime, the product
$\omega_0(x)\omega_0(y)$ grows too slowly near the origin to exploit
mass conservation, and only a local-in-time bound for the zeroth moment
is obtained.

\item[\textbf{Case 2.}]
$\ell \geq \tfrac{1}{2}$,
\hfill\refstepcounter{equation}
(\theequation)\label{eq:Case2}

\noindent
so that $2\ell \geq 1$. In this super-linear regime, $y^{2\ell} \le y$ for
$y \in (0,1)$, which allows mass conservation to absorb one power of
the zeroth moment and yields a uniform bound on every finite time interval.
\end{itemize}

\begin{lemma}\label{Lem:ZM}
Let $T > 0$.
\begin{subequations}
\begin{itemize}
\item[\upshape(a)]
Assume \eqref{eq:Case1}. Then there exists
$C_1(T,f^{\mathrm{in}}) > 0$, independent of $n$, such that
\begin{align}
\int_0^n f_n(t,x)\,dx
\le C_1(T,f^{\mathrm{in}}),
\qquad t \in [0,T].
\label{eq:ZM-1}
\end{align}
\item[\upshape(b)]
Assume \eqref{eq:Case2}. Then there exists $C_1(T) > 0$,
independent of $n$, such that
\begin{align}
\int_0^n f_n(t,x)\,dx \le C_1(T),
\qquad t \in [0,T].
\label{eq:ZM-2}
\end{align}
\end{itemize}
\end{subequations}
\end{lemma}

\begin{proof}
Set
\[
M_0^n(t) := \int_0^n f_n(t,x)\,dx,
\qquad
M_1^n(t) := \int_0^n x\,f_n(t,x)\,dx,
\]
and define the intermediate quantity
\[
\mathcal{J}(s)
:= \int_0^1 y^{2\ell}\,f_n(s,y)\,dy.
\]
Taking $g \equiv 1$ in \eqref{Tr:WF} and applying \eqref{NOP} gives
\begin{align}
M_0^n(t)
\le M_0^n(0)
+ (\beta_0 - 1)
\int_0^t\!\int_0^n\!\int_0^n
a(y,z)\,f_n(s,y)\,f_n(s,z)
\,dy\,dz\,ds.
\label{eq:M0-basic}
\end{align}
We decompose $(0,n)^2$ according to the kernel structure
\eqref{eq:kernel-class} and write
\[
\int_0^t\!\int_0^n\!\int_0^n
a\,f_n\,f_n\,dy\,dz\,ds
= I_{00}(t) + 2I_{01}(t) + I_{11}(t),
\]
where
\begin{align*}
I_{00}(t)
&:= A_0\int_0^t\!\int_0^1\!\int_0^1
\omega_0(y)\,\omega_0(z)\,f_n(s,y)\,f_n(s,z)
\,dy\,dz\,ds,\\
I_{01}(t)
&:= A_0\int_0^t\!\int_0^1\!\int_1^n
\omega_0(y)\,\omega_\infty(z)\,f_n(s,y)\,f_n(s,z)
\,dy\,dz\,ds,\\
I_{11}(t)
&:= A_0\int_0^t\!\int_1^n\!\int_1^n
\omega_\infty(y)\,\omega_\infty(z)\,f_n(s,y)\,f_n(s,z)
\,dy\,dz\,ds.
\end{align*}
We first bound all three integrals in terms of $\mathcal{J}(s)$,
deferring the case distinction to the final step.

\medskip
\noindent\textit{Estimate of $I_{11}$.}
Since $a(y,z) = A_0\,\omega_\infty(y)\,\omega_\infty(z)$ factors
for $y,z \ge 1$,
\[
I_{11}(t)
= A_0\int_0^t
\Bigl(\int_1^n \omega_\infty(y)\,f_n(s,y)\,dy\Bigr)^{\!2}ds.
\]
Applying \eqref{eq:tail3} with $m = 1$ and $g \in \mathcal{G}_b$
satisfying $g(1) > 0$, we obtain
\begin{align}
I_{11}(t)
\le \frac{A_0\,C_0}{\theta\,g(1)}.
\label{est:I11}
\end{align}
\medskip
\noindent\textit{Estimate of $I_{00}$.}
Using \eqref{eq:w0-growth},
$\omega_0(y)\,\omega_0(z) \le A_1^2\,y^\ell z^\ell$. By Young's inequality
$y^\ell z^\ell \le \tfrac{1}{2}(y^{2\ell}+z^{2\ell})$
and symmetry in $y$ and $z$,
\[
I_{00}(t)
\le A\,
\int_0^t
\Bigl(\int_0^1 y^\ell f_n(s,y)\,dy\Bigr)^{\!2}ds.
\]
Applying the Cauchy--Schwarz inequality,
\begin{align}
\Bigl(\int_0^1 y^\ell f_n(s,y)\,dy\Bigr)^{\!2}
\le
\Bigl(\int_0^1 y^{2\ell} f_n(s,y)\,dy\Bigr)
\Bigl(\int_0^1 f_n(s,y)\,dy\Bigr)
\le \mathcal{J}(s)\,M_0^n(s),
\label{eq:ZM-CSI}
\end{align}
and therefore
\begin{align}
I_{00}(t)
\le A_0\, A_1^2
\int_0^t \mathcal{J}(s)\,M_0^n(s)\,ds.
\label{est:I00}
\end{align}

\medskip
\noindent\textit{Estimate of $I_{01}$.}
By \eqref{eq:w0-growth} and \eqref{eq:kernel-class},
\[
I_{01}(t)
\le A\,
\int_0^t
\Bigl(\int_0^1 y^\ell f_n(s,y)\,dy\Bigr)
\Bigl(\int_1^n \omega_\infty(z)\,f_n(s,z)\,dz\Bigr)
ds.
\]
Applying the AM-GM inequality $ab \le \tfrac{1}{2}(a^2+b^2)$ and then
Cauchy--Schwarz as above,
\begin{align}
I_{01}(t)
&\le A\,
\int_0^t \mathcal{J}(s)\,M_0^n(s)\,ds
+ A\,
\int_0^t
\Bigl(\int_1^n \omega_\infty(z)\,f_n(s,z)\,dz
\Bigr)^{\!2}ds
\notag\\
&\le A\,
\int_0^t \mathcal{J}(s)\,M_0^n(s)\,ds
+ \frac{A \,C_0}{\theta\,g(1)},
\label{est:I01}
\end{align}
where \eqref{est:I11} was used in the last step.
Substituting \eqref{est:I00}, \eqref{est:I01},
and \eqref{est:I11} into \eqref{eq:M0-basic},
we obtain the unified integral inequality
\begin{align}
M_0^n(t)
\le M_0^n(0)
+ A\,\int_0^t \mathcal{J}(s)\,M_0^n(s)\,ds
+ \frac{A\,C_0}{\theta\,g(1)},
\label{eq:unified}
\end{align}
It remains to estimate $\mathcal{J}(s)$, which is where the two cases diverge.

\medskip
\noindent\textit{Case (a), $\ell \in (0,\frac{1}{2})$, so $2\ell < 1$.}
Since $y \in (0,1)$ and $2\ell < 1$, we have $y^{2\ell} \le 1$, so that
\[
\mathcal{J}(s)
= \int_0^1 y^{2\ell} f_n(s,y)\,dy
\le \int_0^1 f_n(s,y)\,dy
\le M_0^n(s).
\]
Substituting into \eqref{eq:unified} yields the Riccati-type integral
inequality
\begin{align}
M_0^n(t)
\le M_0^n(0)
+ A\,\int_0^t \bigl(M_0^n(s)\bigr)^2\,ds
+ \frac{A\,C_0}{\theta\,g(1)}.
\label{eq:Riccati}
\end{align}
Since $M_0^n(0) \le \|f^{\mathrm{in}}\|_{\Xi_0} < \infty$, a standard
comparison argument for \eqref{eq:Riccati} shows that $M_0^n$ remains
bounded on $[0,T]$ with a constant depending on $T$ and $f^{\mathrm{in}}$.
Hence there exists $C_1(T,f^{\mathrm{in}}) > 0$, independent of $n$,
satisfying \eqref{eq:ZM-1}.

\medskip
\noindent\textit{Case (b), $\ell > \frac{1}{2}$, so $2\ell > 1$.}
Since $y \in (0,1)$ and $2\ell > 1$, we have $y^{2\ell} \le y$, so that
\[
\mathcal{J}(s)
= \int_0^1 y^{2\ell} f_n(s,y)\,dy
\le \int_0^1 y\,f_n(s,y)\,dy
\le M_1^n(s)
\le \Theta,
\]
where we used \eqref{eq:mass-ineq} in the last step.
Substituting into \eqref{eq:unified} yields the linear integral inequality
\begin{align}
M_0^n(t)
\le M_0^n(0)
+  A\,\Theta\int_0^t M_0^n(s)\,ds
+ \frac{A \,C_0}{\theta\,g(1)}.
\label{eq:Gronwall}
\end{align}
An application of Gr\"{o}nwall's inequality to \eqref{eq:Gronwall} gives
\[
M_0^n(t)
\le \Bigl(M_0^n(0) + \frac{A\,C_0}{\theta\,g(1)}\Bigr)
e^{ A\,\Theta\,T},
\qquad t \in [0,T],
\]
and since $M_0^n(0) \le \|f^{\mathrm{in}}\|_{\Xi_0}<\infty$, there exists
$C_1(T) > 0$, independent of $n$, satisfying \eqref{eq:ZM-2}.
This completes the proof.
\end{proof}

\medskip

Next, we turn to uniform integrability. For $n \geq 1$, $a \in (1,n]$,
$\delta \in (0,+\infty)$, and $t \in \mathbb{R}_{+}$, define
\[
W_{a,\delta}^{n}(t) = \sup \left\{
\int_{0}^{a} \mathbf{1}_{E}(x)\, f_{n}(t,x)\, dx
\;\Big|\;
E\subset \mathbb{R}_{+} \text{ measurable},\; |E| \leq \delta
\right\}.
\]

\begin{lemma}\label{Lem:UI}
Let $T \in (0,+\infty)$ and $a \in (1,+\infty)$.
Then for every $n \ge 1$, $t \in [0,T]$, and $\delta \in (0,+\infty)$:
\begin{itemize}
\item[\upshape(a)]
If $\ell \in \bigl(0,\,\frac{3p-2}{2p}\bigr)$,
there exists $C_2(T,f^{\mathrm{in}}) > 0$,
independent of $n$, such that
\begin{align}
W_{a,\delta}^n(t) \le W_{a,\delta}^n(0)+ C_2(T,f^{\mathrm{in}})\delta^{(p-1)/p}.
\label{eq:UI-1}
\end{align}
\item[\upshape(b)]
If $\ell \ge \frac{3p-2}{2p}$,
there exists $C_2(T) > 0$,
independent of $n$ and $f^{\mathrm{in}}$, such that
\begin{align}
W_{a,\delta}^n(t)
\le
W_{a,\delta}^n(0)+ C_2(T) \delta^{(p-1)/p}.
\label{eq:UI-2}
\end{align}
\end{itemize}
\end{lemma}

\begin{proof}
Let $E \subset (0,a)$ be measurable with $|E| \le \delta$.
Taking $g = \mathbf{1}_E$ in \eqref{Tr:WF} and discarding the
non-positive contribution
$-\int_0^t\iint \mathbf{1}_E(y)\,a\,f_nf_n \le 0$,
we obtain
\begin{align}
\int_E f_n(t,x)\,dx
\le \int_E f_n^{\mathrm{in}}\,dx
+ \int_0^t\!\int_0^n\!\int_0^n
\left(\int_0^y \mathbf{1}_E(x)\,b(x,y,z)\,dx\right)
a(y,z)\,f_n f_n\,dy\,dz\,ds.
\label{eq:UI-ineq}
\end{align}
Applying H\"{o}lder's inequality with exponents $p/(p-1)$ and $p$,
together with $|E \cap (0,y)| \le \delta$ and \eqref{eq:p-cond},
\begin{align}
\int_0^y \mathbf{1}_E(x)\,b(x,y,z)\,dx
\le \left(\frac{B_p}{2}\right)^{\!1/p}
\delta^{(p-1)/p}\,y^{(1-p)/p}.
\label{eq:Hld-bnd}
\end{align}
Substituting \eqref{eq:Hld-bnd} into \eqref{eq:UI-ineq} yields
\begin{align}
\int_E f_n(t,x)\,dx
\le \int_E f_n^{\mathrm{in}}\,dx
+ \left(\frac{B_p}{2}\right)^{\!1/p}
\delta^{(p-1)/p}\,\mathcal{I}(t),
\label{eq:UI-main}
\end{align}
where
\[
\mathcal{I}(t)
:= \int_0^t\!\int_0^n\!\int_0^n
y^{(1-p)/p}\,a(y,z)\,f_n(s,y)\,f_n(s,z)
\,dy\,dz\,ds.
\]
We claim that $\sup_{n\ge1}\mathcal{I}(T)<\infty$.
Decomposing $(0,n)^2$ via \eqref{eq:kernel-class},
we write
$\mathcal{I}(t) = \mathcal{I}_{00}(t)
+\mathcal{I}_{01}(t)+\mathcal{I}_{10}(t)
+\mathcal{I}_{11}(t)$,
corresponding to the regions $(0,1)^2$,
$(0,1)\times(1,n)$, $(1,n)\times(0,1)$, and $(1,n)^2$. Define
\begin{align*}
\mathcal{J}(s) := \int_0^1 y^{2\ell}
f_n(s,y)\,dy,
\qquad
\mathcal{K}(s) := \int_0^1 y^{2(1-p)/p+2\ell}
f_n(s,y)\,dy.
\end{align*}
By the Cauchy--Schwarz inequality with measure $f_n(s,y)\,dy$,
\begin{align}
\left(\int_0^1 y^{\ell}
f_n\,dy\right)^{\!2}
\le \mathcal{J}(s)\,M_0^n(s),
\qquad
\left(\int_0^1 y^{(1-p)/p+\ell}
f_n\,dy\right)^{\!2}
\le \mathcal{K}(s)\,M_0^n(s).
\label{eq:CS-JK}
\end{align}
Since $(1-p)/p < 0$, the weight satisfies $y^{(1-p)/p} \le 1$
for all $y \ge 1$, so the regions $\mathcal{I}_{10}$ and $\mathcal{I}_{11}$
are straightforward. For $\mathcal{I}_{11}$, using
$a(y,z) = A_0\,\omega_\infty(y)\omega_\infty(z)$ yields
\begin{align}
\mathcal{I}_{11}(t)
= A_0\int_0^t
\left(\int_1^n\omega_\infty(y)
f_n(s,y)\,dy\right)^{\!2}ds
\le \frac{A_0\,C_0}{\theta\,g(1)}.
\label{est:I11-UI}
\end{align}
For $\mathcal{I}_{10}$, using $y^{(1-p)/p} \le 1$,
$a(y,z) = A_0\,\omega_\infty(y)\omega_0(z)$,
\eqref{eq:w0-growth}, Young's inequality, the first bound in \eqref{eq:CS-JK}, and \eqref{eq:tail3},
\begin{align}
\mathcal{I}_{10}(t)
&\le A\,
\int_0^t
\left(\int_1^n\omega_\infty f_n\,dy\right)
\left(\int_0^1 z^\ell f_n\,dz\right)ds
\notag\\
&\le A\,
\int_0^t\mathcal{J}(s)\,M_0^n(s)\,ds
+ \frac{A\,C_0}{\theta\,g(1)}.
\label{est:I10-UI}
\end{align}
For the regions involving $y\in(0,1)$, the weight $y^{(1-p)/p}$ is
unbounded as $y\to0$ and must be absorbed using the second bound in
\eqref{eq:CS-JK}. For $\mathcal{I}_{01}$, using
$a(y,z)=A_0\,\omega_0(y)\omega_\infty(z)$,
\eqref{eq:w0-growth}, Young's inequality,
\eqref{eq:CS-JK}, and \eqref{eq:tail3},
\begin{align}
\mathcal{I}_{01}(t)
&\le A\,
\int_0^t\left(\int_0^1 y^{(1-p)/p+\ell}
f_n\,dy\right)
\left(\int_1^n\omega_\infty f_n\,dz\right)ds
\notag\\
&\le A\,
\int_0^t\mathcal{K}(s)\,M_0^n(s)\,ds
+\frac{A\,C_0}{\theta\,g(1)} .
\label{est:I01-UI}
\end{align}
For $\mathcal{I}_{00}$, using \eqref{eq:kernel-class},
\eqref{eq:w0-growth}, Young's inequality
$y^\ell z^\ell \le \frac{1}{2}(y^{2\ell}+z^{2\ell})$,
symmetry in $y$ and $z$, and both bounds in \eqref{eq:CS-JK},
\begin{align}
\mathcal{I}_{00}(t)
&\le A\int_0^t
\left(\int_0^1 y^{(1-p)/p+\ell}f_n\,dy\right)
\left(\int_0^1 z^\ell f_n\,dz\right)ds
\notag\\
&\le A
\int_0^t
\bigl[\mathcal{K}(s)+\mathcal{J}(s)\bigr]
M_0^n(s)\,ds.
\label{est:I00-UI}
\end{align}
Collecting \eqref{est:I11-UI}--\eqref{est:I00-UI},
\begin{align}
\mathcal{I}(t)
\le A\,\int_0^t
\bigl[\mathcal{K}(s)+\mathcal{J}(s)\bigr]
M_0^n(s)\,ds +\frac{A\,C_0}{\theta\,g(1)}.
\label{est:I-unified}
\end{align}

It remains to bound $\mathcal{J}(s)$ and $\mathcal{K}(s)$.
Applying the same argument as Lemma~\ref{Lem:ZM}: for $y \in (0,1)$,
$y^\alpha \le y$ if $\alpha \ge 1$ and $y^\alpha \le 1$ if $\alpha \ge 0$,
with the exponent of $\mathcal{K}$ shifted downward by $2(p-1)/p > 0$
relative to that of $\mathcal{J}$, due to the weight $y^{(1-p)/p}$.

In Case~(a), $\ell \in (0,(3p-2)/(2p))$ implies $2(1-p)/p+2\ell < 1$,
so $y^{2(1-p)/p+2\ell} \le 1$ for $y \in (0,1)$ and hence
$\mathcal{K}(s) \le M_0^n(s)$.
If $\ell \le 1/2$ then $\mathcal{J}(s) \le M_0^n(s)$;
if $\ell > 1/2$ then $\mathcal{J}(s) \le \Theta$.
In either sub-case,
\[
\bigl[\mathcal{K}(s)+\mathcal{J}(s)\bigr]M_0^n(s)
\le 2\,(M_0^n(s))^2 + \Theta\,M_0^n(s),
\]
and since $M_0^n(s) \le C_1(T,f^{\mathrm{in}})$ for all $s\in[0,T]$
from Lemma~\ref{Lem:ZM}(a), we obtain
\[
2\,(M_0^n(s))^2+\Theta\,M_0^n(s)
\le C_1(T,f^{\mathrm{in}})\bigl(2\,C_1(T,f^{\mathrm{in}})+\Theta\bigr).
\]
Inserting into \eqref{est:I-unified} and integrating over $[0,T]$ gives
\begin{align}
\mathcal{I}(T)
&\le
A\,T\,
C_1(T,f^{\mathrm{in}})
\bigl(2\,C_1(T,f^{\mathrm{in}})+\Theta\bigr)
+\frac{A\,C_0}{\theta\,g(1)}
\notag\\
&=: \Lambda_T(f^{\mathrm{in}}) < \infty.
\label{eq:I-bnd-a}
\end{align}
%%%%%%%%%%%%%%%%%%%%%%%%%%%%%%%%%%%%%%%%%%%%%%%%%%%%%%%%%%%%%%%%%%%%%%%%%%%%%%%%%%%%%%%%%%%%%%%%%%%%%%%%%%%%%%%%%%%%%%%%%%%%%%%%%%%%%%%%%%%%%%%%%%%%%%%%%%%%%%%%%%%%%%%%%%%%%%%%%%%%%%%%
In Case~(b), $\ell \ge (3p-2)/(2p) > 1/2$ ensures both $2\ell \ge 1$
and $2(1-p)/p+2\ell \ge 1$, so $y^{2\ell} \le y$ and
$y^{2(1-p)/p+2\ell} \le y$ for $y \in (0,1)$, giving
$\mathcal{J}(s) \le \Theta$ and $\mathcal{K}(s) \le \Theta$.
Inserting into \eqref{est:I-unified} and using $M_0^n \le C_1(T)$
from Lemma~\ref{Lem:ZM}(b) gives
\begin{align}
\mathcal{I}(T)
\le 2A\,\Theta\,C_1(T)\,T + \frac{A\,C_0}{\theta\,g(1)}
=: \Lambda_T < \infty,
\label{eq:I-bnd-b}
\end{align}
where $\Lambda_T$ is independent of $n$ and $f^{\mathrm{in}}$.

In both cases the bound takes the form $\mathcal{I}(T) \le \Lambda_T$,
where
\[
\tilde{\Lambda}_T :=
\begin{cases}
\Lambda_T(f^{\mathrm{in}}) & \text{in Case~(a)},\\
\Lambda_T                  & \text{in Case~(b)},
\end{cases}
\]
is a finite constant independent of $n$.
Taking the supremum over all measurable $E \subset (0,a)$
with $|E| \le \delta$ in \eqref{eq:UI-main} gives
\[
W_{a,\delta}^n(t)
\le W_{a,\delta}^n(0)
+ \left(\frac{B_p}{2}\right)^{\!1/p}
\tilde{\Lambda}_T\,\delta^{(p-1)/p},
\qquad t \in [0,T],
\]
from which \eqref{eq:UI-1} or \eqref{eq:UI-2} follows upon setting
$C_2 := (B_p/2)^{1/p}\tilde{\Lambda}_T$.
The convergence $W_{a,\delta}^n(t) \to 0$ as $\delta \to 0$ holds since
$(p-1)/p \in (0,1)$ and $\{f_n^{\mathrm{in}}\}$ is uniformly integrable
in $L^1(0,a)$ by assumption.
\end{proof}

\begin{remark}
The threshold $\ell = (3p-2)/(2p)$ is strictly greater than $1/2$ for all
$p \in (1,2)$, tending to $1/2$ as $p \to 1^+$ and to $1$ as $p \to 2^-$.
It arises from the condition $2(1-p)/p+2\ell \ge 1$, which ensures
$\mathcal{K}(s) \le M_1^n(s) \le \Theta$ via mass conservation--exactly the same mechanism as in Lemma~\ref{Lem:ZM}, but with the exponent $2\ell$ replaced by $2(1-p)/p + 2\ell$ due to the weight $y^{(1-p)/p}$ inherited from \eqref{eq:p-cond}.
\end{remark}

\begin{lemma}\label{Lem:time-eqc}
Let $T > 0$ and $\lambda \in (1,n)$. Then for every $0 \le s \le t \le T$
there exists a constant $C_3(T) > 0$, independent of $n$, such that
\begin{align}
\int_0^\lambda |f_n(t,x) - f_n(s,x)|\,dx
\le C_3(T)\,(t-s).
\label{eq:time-equicont}
\end{align}
\end{lemma}

\begin{proof}
It suffices to show that $\|\partial_t f_n(t)\|_{L^1(0,\lambda)}$ is
uniformly bounded on $[0,T]$ independently of $n$.
Splitting $\partial_t f_n$ into its breakage-gain and
breakage-loss contributions and integrating over $(0,\lambda)$,
we obtain via Fubini's theorem
\begin{align}
\int_0^\lambda |\partial_t f_n(t,x)|\,dx
&\le \int_0^n\!\int_0^n
\left(\int_0^{\min(y,\lambda)} b(x,y,z)\,dx\right)
a(y,z)\,f_n(t,y)\,f_n(t,z)\,dy\,dz
\notag\\
&\quad
+ \int_0^\lambda f_n(t,y)\,\int_0^n a(y,z)\,f_n(t,z)\,dz\,dy.
\notag
\end{align}
Applying the fragment-number bound \eqref{NOP} to the first term,
both contributions are controlled as
\begin{align}
\int_0^\lambda |\partial_t f_n(t,x)|\,dx
\le (\beta_0+1)
\int_0^n\!\int_0^n
a(y,z)\,f_n(t,y)\,f_n(t,z)\,dy\,dz.
\label{eq:tec-reduced}
\end{align}
It remains to bound the double integral in \eqref{eq:tec-reduced}
uniformly in $t \in [0,T]$. Decomposing $(0,n)^2$ via
\eqref{eq:kernel-class} into the four standard regions and applying,
in each, the identical Cauchy--Schwarz and Young arguments used in
Lemma~\ref{Lem:ZM}, together with the bounds \eqref{eq:ZM-1}--\eqref{eq:ZM-2}
(depending on the behaviour of the kernel near zero) and the
weighted-moment bounds from Lemma~\ref{Lem:Tail}, one obtains a constant
$K(T) > 0$, independent of $n$ and $t$, such that
\begin{align}
\int_0^n\!\int_0^n
a(y,z)\,f_n(t,y)\,f_n(t,z)\,dy\,dz
\le K(T),
\qquad t \in [0,T].
\label{eq:collision-bnd}
\end{align}
Substituting \eqref{eq:collision-bnd} into \eqref{eq:tec-reduced} yields
$\|\partial_t f_n(t)\|_{L^1(0,\lambda)} \le (\beta_0+1)\,K(T)$
for all $t \in [0,T]$, and integrating over $[s,t]$ gives
\eqref{eq:time-equicont} with $C_3(T) := (\beta_0+1)\,K(T)$.
\end{proof}

We shall make use of the following continuity result for bilinear integral operators with respect to weak convergence in $L^1$ spaces. Its proof follows along the same lines as that of \cite[Lemma 4.1]{Stewart1989}, and is therefore omitted.

\begin{lemma}\label{lem:bilinear}
Let $0<a\le b<\infty$ and let
\[
\Omega\in L^\infty\bigl((0,a)\times(0,b)\bigr).
\]
Define the bilinear functional
\[
\Upsilon(f,g)
:=
\int_0^a\int_0^b
\Omega(x,y)\,f(x)\,g(y)\,dy\,dx,
\qquad
(f,g)\in L^1(0,a)\times L^1(0,b).
\]
Suppose that $(f_n)$ converges weakly to $f$ in $L^1(0,a)$ and that
$(g_n)$ converges weakly to $g$ in $L^1(0,b)$. Then
\[
\lim_{n\to\infty}\Upsilon(f_n,g_n) =\Upsilon(f,g).
\]
\end{lemma}
%%%%%%%%%%%%%%%%%%%%%%%%%%%%%%%%%%%%
\section{\textbf{Proof of Theorem~\ref{Thm:GE}}}\label{sec:proof}
%%%%%%%%%%%%%%%%%%%%%%%%%%%%%%%%%%%%

\begin{proof}
Fix $T>0$ and $a\in(1,n)$. We collect the uniform-in-$n$ estimates available on $[0,T]$.
From Lemma~\ref{Lem:ZM}(b),
\begin{align}
\sup_{n\ge1}\sup_{t\in[0,T]}
\int_0^a f_n(t,x)\,dx
\le
\sup_{n\ge1}\sup_{t\in[0,T]} M_0^n(t)
\le C_1(T).
\label{eq:unif-L1}
\end{align}
From Lemma~\ref{Lem:time-eqc}, for all $n>a$ and $s,t\in[0,T]$,
\begin{align}
\int_0^a |f_n(t,x)-f_n(s,x)|\,dx
\le C_3(T)\,|t-s|,
\label{eq:unif-Lip}
\end{align}
so the family $\{f_n\}_{n\ge1}$ is equicontinuous from $[0,T]$ into $L^1(0,a)$.
From Lemma~\ref{Lem:UI}(b), for every $\varepsilon>0$ there exists $\delta>0$
such that for every measurable set $E\subset(0,a)$ with $|E|\le\delta$ and all
$n\ge1$, $t\in[0,T]$,
\begin{align}
\int_E f_n(t,x)\,dx<\varepsilon.
\label{eq:unif-int}
\end{align}

The uniform bound \eqref{eq:unif-L1} together with the uniform integrability
property \eqref{eq:unif-int}, applied on the bounded domain $(0,a)$, implies by
the Dunford--Pettis theorem that, for each fixed $t\in[0,T]$, the set
$\{f_n(t,\cdot)\}_{n\ge1}$ is relatively weakly compact in $L^1(0,a)$.
Combined with the equicontinuity \eqref{eq:unif-Lip}, this yields relative
compactness in $\mathcal{C}([0,T];L^1_w(0,a))$. Consequently, by a diagonal extraction,
there exists a subsequence (not relabelled) and a limit function
$f:[0,\infty)\times(0,\infty)\to\mathbb{R}$ such that
\begin{align}
f_n \longrightarrow f
\quad\text{in } \mathcal{C}([0,T];L^1_w(0,a)),
\quad \text{for every } a>1 \text{ and } T>0.
\label{eq:weak-conv}
\end{align}
Since $f(t,\cdot)$ is obtained as the weak limit of a sequence of non-negative functions, it follows that $f(t,x)\ge0$  for a.e. in $(0,\infty)$ for every $t\in\mathbb{R}_+$.

Passing to the limit $n \to \infty$ in the uniform estimates for $f_n$, we obtain the following properties of $f$. From Corollary~\ref{COR1}, the truncated solutions satisfy
\begin{align*}
\int_0^n g(x)\,f_n(t,x)\,dx \le \frac{C_0}{\theta},
\qquad t \ge 0,\quad n \ge 1.
\end{align*}
For any fixed $a > 1$, since $g \in L^\infty(0,a)$
and $f_n(t) \rightharpoonup f(t)$ weakly in $L^1(0,a)$
by \eqref{eq:weak-conv}, we may pass to the limit to obtain
\begin{align*}
\int_0^a g(x)\,f(t,x)\,dx
= \lim_{n \to \infty} \int_0^a g(x)\,f_n(t,x)\,dx
\le \limsup_{n \to \infty} \int_0^n g(x)\,f_n(t,x)\,dx
\le \frac{C_0}{\theta}.
\end{align*}
Letting $a \to \infty$ and applying the monotone convergence
theorem (since $g \ge 0$ and $f \ge 0$),
\begin{align}
\int_0^\infty g(x)\,f(t,x)\,dx
\le \frac{C_0}{\theta},
\qquad t \ge 0,
\label{eq:lim-HMBD}
\end{align}
so $f \in L^\infty\bigl((0,\infty);\,\Xi_g^+\bigr)$.
%%%%%%%%%%%%%%%%%%%%%%%%%%%%%%%%%%%%%%%%%%%%%%%%%%%%%%%%%%%%%%%%%%%%%%%%%%%%%%%%%%%%%%%%%%%%%%%%%%%%%%%%%%%%%%%%%%%%%%%%%%%%%%%%%%%%%%%%%%%%%%%%%%%%%%%%%%%%%%%%%%%%%%%%%%%%%%%%%%%%%%%%
Next, we show that the estimate \eqref{eq:HMBD2} carries over to the
limit function $f$, namely
\begin{align}
\int_0^t\!\int_0^\infty\!\int_0^\infty
g(y)\,a(y,z)\,f(s,y)\,f(s,z)\,dy\,dz\,ds
\le \frac{C_0}{\theta},
\qquad t\ge0.
\label{eq:HMBD2-lim}
\end{align}
Fix $a>1$ and $s\in[0,t]$. Since $g(\cdot)\,a(\cdot,z)\in L^\infty(0,a)$
for each $z\in(0,a)$, the kernel
$\Omega(y,z):=g(y)\,a(y,z)$ belongs to $L^\infty((0,a)\times(0,a))$.
By the weak convergence $f_n(s)\rightharpoonup f(s)$ in $L^1(0,a)$
from \eqref{eq:weak-conv} (which holds for every $s\in[0,t]$),
Lemma~\ref{lem:bilinear} gives, for each fixed $s\in[0,t]$,
\[
\int_0^a\!\int_0^a g(y)\,a(y,z)\,f_n(s,y)\,f_n(s,z)\,dy\,dz
\;\xrightarrow{n\to\infty}\;
\int_0^a\!\int_0^a g(y)\,a(y,z)\,f(s,y)\,f(s,z)\,dy\,dz.
\]

Define
\[
h_n(s)
:=
\int_0^a\!\int_0^a g(y)\,a(y,z)\,f_n(s,y)\,f_n(s,z)\,dy\,dz,
\]
\[
h(s)
:=
\int_0^a\!\int_0^a g(y)\,a(y,z)\,f(s,y)\,f(s,z)\,dy\,dz.
\]
Since $h_n(s)\to h(s)$ for every $s\in[0,t]$ and $h_n(s)\ge0$ for all $s\in[0,t]$ (because $g$, $a$, and $f_n$ are non-negative), it follows from Fatou's lemma that
\begin{align*}
\int_0^t h(s)\,ds
= \int_0^t \lim_{n\to\infty} h_n(s)\,ds
= \int_0^t \liminf_{n\to\infty} h_n(s)\,ds
\le \liminf_{n\to\infty}\int_0^t h_n(s)\,ds.
\end{align*}
Since \begin{align*}
h_n(s)\le \int_0^n\!\int_0^n g\,a\,f_n f_n\,dy\,dz
\end{align*}
for all $n\ge a$, the bound \eqref{eq:HMBD2} gives \begin{align*}
\int_0^t h_n(s)\,ds\le \frac{C_0}{\theta},
\end{align*}
 and therefore
\begin{align}
\int_0^t\!\int_0^a\!\int_0^a
g(y)\,a(y,z)\,f(s,y)\,f(s,z)\,dy\,dz\,ds
\le \frac{C_0}{\theta}.
\label{eq:HMBD2-a}
\end{align}

Since $g(y)\,a(y,z)\,f(s,y)\,f(s,z)\ge0$, the left-hand side of
\eqref{eq:HMBD2-a} is non-decreasing in $a$. Letting $a\to\infty$
and applying the monotone convergence theorem, the claim \eqref{eq:HMBD2-lim} hold.
%%%%%%%%%%%%%%%%%%%%%%%%%%%%%%%%%%%%%%%%%%%%%%%%%%%%%%%%%%%%%%%%%%%%%%%%%%%%%%%%%%%%%%%%%%%%%%%%%%%%%%%%%%%%%%%%%%%%%%%%%%%%%%%%%%%%%%%%%%%%%%%%%%%%%%%%%%%%%%%%%%%%%%%%%%%%%%%%%%%%%%%%
From \eqref{eq:mass-ineq}, for any $a > 0$,
\begin{align*}
\int_0^a x\,f(t,x)\,dx
= \lim_{n\to\infty}\int_0^a x\,f_n(t,x)\,dx \le \Theta,
\end{align*}
so \eqref{eq:mass-ineq-limit} holds by monotone convergence.

From \eqref{eq:tail1} of Lemma~\ref{Lem:Tail} with $m > 1$ and the
convergence \eqref{eq:weak-conv}, for any $a > m$,
\begin{align*}
\int_m^a g(x)\,f(t,x)\,dx
= \lim_{n\to\infty}
\int_m^a g(x)\,f_n(t,x)\,dx
\le \int_m^\infty g(x)\,f^{\mathrm{in}}(x)\,dx.
\end{align*}
Letting $a \to \infty$ by monotone convergence, we obtain
\begin{align}
\int_m^\infty g(x)\,f(t,x)\,dx \le C_0,
\qquad t \ge 0,\quad m > 1.
\label{eq:tail-lim}
\end{align}
Since $g$ is non-decreasing on $(0,\infty)$ (which follows from
$g(x)/x$ non-decreasing and $g \ge 0$), we have $g(x) \ge g(m)$ for
$x \ge m$, so
\begin{align}
\int_m^\infty f(t,x)\,dx
\le \frac{C_0}{g(m)} \to 0
\quad \text{as } m \to \infty,
\label{eq:tail-f}
\end{align}
uniformly in $t \ge 0$.

From \eqref{eq:tail3} of Lemma~\ref{Lem:Tail} with $m > 1$, for any
$a > m$ and $T > 0$,
\begin{align*}
\int_0^T\!\left(\int_m^a
\omega_\infty(y)\,f_n(s,y)\,dy\right)^{\!2}\!ds
\le \frac{C_0}{\theta\,g(m)}.
\end{align*}
Applying Fatou's lemma as $n \to \infty$ and then $a \to \infty$,
\begin{align}
\int_0^T\!\left(\int_m^\infty
\omega_\infty(y)\,f(s,y)\,dy\right)^{\!2}\!ds
\le \frac{C_0}{\theta\,g(m)},
\qquad m > 1,\quad T > 0.
\label{eq:tail3-lim}
\end{align}

We now claim that 
\begin{align}
f \in \mathcal{C}([0,\infty);\Xi_0).\label{eq:strong-conv}
\end{align}
 For any $t,s \in [0,T]$ and $m > 1$,
\begin{align*}
\|f_n(t)-f_n(s)\|_{\Xi_0}
&\le \int_0^m |f_n(t)-f_n(s)|\,dx
+ \int_m^\infty f_n(t,x)\,dx
+ \int_m^\infty f_n(s,x)\,dx.
\end{align*}
Using \eqref{eq:unif-Lip}, and \eqref{eq:tail-f}, we deduce
\begin{align*}
\|f_n(t)-f_n(s)\|_{\Xi_0}
\le C_3(T)|t-s| + \frac{2\,C_0}{g(m)}.
\end{align*}
Given $\varepsilon > 0$, first choose $m$ large enough so that
$2C_0/g(m) < \varepsilon/2$ (possible since $C_0/g(m) \to 0$), then
choose $|t-s| < \varepsilon/(2C_3(T))$. This yields
\begin{align}
\|f_n(t)-f_n(s)\|_{\Xi_0}<\varepsilon, \label{eq:strong-conv-fn}
\end{align}
and therefore the family $(f_n)$ is equicontinuous in
$L^1(0,\infty)$ endowed with its strong topology. Since $f_n(t)-f_n(s)$ converges weakly to $f(t)-f(s)$ in $L^1(0,\infty)$, it follows from \eqref{eq:strong-conv-fn} and the weak lower
semicontinuity of the $L^1$-norm that
\begin{align*}
\|f(t)-f(s)\|_{\Xi_0}\le \varepsilon,
\end{align*}
which proves \eqref{eq:strong-conv}. Fix $\psi \in L^\infty(0,\infty)$ and $t > 0$.
By \eqref{NOP} and \eqref{def:tilde-psi},
\begin{align*}
|\tilde\psi| \le (\beta_0+1)\|\psi\|_{L^\infty} =: C_\psi.
\end{align*}
The truncated weak formulation \eqref{Tr:WF} rearranges to
\begin{align}
\int_0^n\!\psi(x)\bigl(f_n(t,x)
-f_n^{\mathrm{in}}(x)\bigr)\,dx
=\int_0^t\!\int_0^n\!\int_0^n
\tilde\psi(y,z)\,a(y,z)\,f_n(s,y)\, f_n(s,z)\,dy\,dz\,ds.
\label{eq:tr-wf-pass}
\end{align}
For the left-hand side of \eqref{eq:tr-wf-pass}, for any $m > 1$ we write
\begin{align*}
\int_0^n\!\psi(x)\,(f_n(t,x)-f_n^{\mathrm{in}}(x))\,dx
=&\int_0^m\psi(x)\,(f_n(t,x)-f_n^{\mathrm{in}}(x))\,dx\\
&\qquad +\int_m^n\psi(x)\,(f_n(t,x)-f_n^{\mathrm{in}}(x))\,dx.
\end{align*}
By \eqref{eq:strong-conv}, the first term converges to
$\int_0^m\psi(x)\,(f(t)-f^{\mathrm{in}})\,dx$. For the second, using
$g(x) \ge g(m)$ for $x \ge m$ and Corollary~\ref{COR1},
\begin{align*}
\left|\int_m^n\psi(x)\,f_n(t)\,dx\right|
\le \frac{\|\psi(x)\|_{L^\infty}}{g(m)}
\int_m^n g(x)\,f_n(t,x)\,dx
\le \|\psi\|_{L^\infty}\,\frac{C_0}{\theta\,g(m)},
\end{align*}
and similarly
\begin{align*}
\left|\int_m^n\psi(x)\,f_n^{\mathrm{in}}(x)\,dx\right|
\le \|\psi\|_{L^\infty}\frac{C_0}{g(m)}.
\end{align*}
Both terms tend to $0$ as $m \to \infty$. The same bounds hold for $f$ using
\eqref{eq:lim-HMBD} and \eqref{eq:tail-lim}. Letting $n \to \infty$
and then $m \to \infty$,
\begin{align}
\lim_{n\to\infty}
\int_0^n\!\psi(x)\bigl(f_n(t,x)
-f_n^{\mathrm{in}}(x)\bigr)\,dx
=\int_0^\infty\psi(x)
\bigl(f(t,x)-f^{\mathrm{in}}(x)\bigr)\,dx.
\label{eq:lhs-pass}
\end{align}

Fix $a > 1$ (to be chosen). For $n \ge a$, split the double integral as
\[
\int_0^t\!\int_0^n\!\int_0^n
\tilde\psi(y,z)\,a(y,z)\,f_n(s,y)f_n(s,z)\,dy\,dz\,ds
=\mathcal{T}_1^n(a)
+2\,\mathcal{T}_2^n(a)
+\mathcal{T}_3^n(a),
\]
where
\begin{align*}
\mathcal{T}_1^n(a)
&:=\int_0^t\!\int_0^a\!\int_0^a
\tilde\psi(y,z)\,a(y,z)\,f_n(s,y)f_n(s,z)\,dz\,dy\,ds, \\
\mathcal{T}_2^n(a)
&:=\int_0^t\!\int_0^a\!\int_a^n
\tilde\psi(y,z)\,a(y,z)\,f_n(s,y)f_n(s,z)\,dz\,dy\,ds, \\
\mathcal{T}_3^n(a)
&:=\int_0^t\!\int_a^n\!\int_a^n
\tilde\psi(y,z)\,a(y,z)\,f_n(s,y)f_n(s,z)\,dz\,dy\,ds.
\end{align*}
The quantities $\mathcal{T}_1(a)$, $\mathcal{T}_2(a)$, $\mathcal{T}_3(a)$
are defined by replacing $f_n$ with $f$ and integrating over $(0,a)^2$,
$(0,a)\times(a,\infty)$, and $(a,\infty)^2$ respectively.
%%%%%%%%%%%%%%%%%%%%%%%%%%%%%%%%%%%%%%%%%%%%%%%%%%%%%%%%%%%%%%%%%%%%%%%%%%%%%%%%%%%%%%%%%%%%
Define
\begin{align*}
\mathcal{Z}_1^n(a,s)
:=
\int_0^a\int_0^a
\tilde{\psi}(y,z)\,a(y,z)\,
f_n(s,y)f_n(s,z)\,dz\,dy,
\end{align*}
and
\begin{align*}
\mathcal{Z}_1(a,s)
:=
\int_0^a\int_0^a
\tilde{\psi}(y,z)\,a(y,z)\,
f(s,y)f(s,z)\,dz\,dy.
\end{align*}
Since $a(\cdot,\cdot)\in L^\infty((0,a)\times(0,a))$, it follows from
Lemma~\ref{lem:bilinear} together with the weak convergence
\eqref{eq:weak-conv} that, for each fixed $s\in(0,t)$,
\begin{align*}
\lim_{n\to\infty}\mathcal{Z}_1^n(a,s)
=
\mathcal{Z}_1(a,s).
\end{align*}
Moreover, using \eqref{eq:unif-L1} and the boundedness of
$a\tilde{\psi}$ on $(0,a)\times(0,a)$, we obtain
\begin{align*}
|\mathcal{Z}_1^n(a,s)|
&\le
C_{\psi}\,\|a\|_{L^\infty((0,a)^2)}
\left(\int_0^a f_n(s,y)\,dy\right)
\left(\int_0^a f_n(s,z)\,dz\right) \\
&\le
C_{\psi}\,\|a\|_{L^\infty((0,a)^2)}\,C_1(T)^2,
\end{align*}
uniformly with respect to $n$ and $s\in[0,t]$.
Therefore, by the Lebesgue dominated convergence theorem,
\begin{align}
\lim_{n\to\infty}\mathcal{T}_1^n(a)
&=
\lim_{n\to\infty}
\int_0^t \mathcal{Z}_1^n(a,s)\,ds\notag \\
&=
\int_0^t \mathcal{Z}_1(a,s)\,ds \notag\\
&=
\int_0^t\int_0^a\int_0^a
\tilde{\psi}(y,z)\,a(y,z)\,
f(s,y)f(s,z)\,dz\,dy\,ds. \label{eq:T1-conv}
\end{align}
%%%%%%%%%%%%%%%%%%%%%%%%%%%%%%%%%%%%%%%%%%%%%%%%%%%%%%%%%%%%%%%%%%%%%%%%%%%%%%%%%%%%%%%%%%%%%%%%%%%%%%%%%%%%%%%%%%%%%%%%%%%%%%%%%%%%%%%%%%%%%%%%%%%%%%%%%%%%%%%%%%%%%%%%%%%%%%%%%%%%%%%%

To bound $\mathcal{T}_2^n(a)$, we split $(0,a) = (0,1)\cup(1,a]$ in the
$y$-variable and use the kernel structure \eqref{eq:kernel-class}.
For $y \in (0,1)$, $z \in (a,n)$,
$a(y,z) = A_0\,\omega_0(y)\omega_\infty(z)
\le A_0A_1\, y^\ell\omega_\infty(z)$.
Applying the Cauchy--Schwarz inequality in $s$, then the estimate
\begin{align*}
\bigl(\int_0^1 y^\ell f_n(s,y)\,dy\bigr)^2 \le \mathcal{J}(s)M_0^n(s)
\le \Theta\,C_1(T)
\end{align*}
from Lemma~\ref{Lem:ZM}(b), and \eqref{eq:tail3} with $m=a$,
\begin{align}
\left|\int_0^t\!\int_0^1\!\int_a^n
\tilde\psi\,a\,f_nf_n\right|
&\le A\,C_{\psi}
\sqrt{\int_0^t\!\left(\int_0^1 y^\ell f_n\,dy
\right)^{\!2}\!ds}\,
\sqrt{\int_0^t\!\left(\int_a^n
\omega_\infty f_n\,dz\right)^{\!2}\!ds}
\notag\\
&\le A\,C_{\psi}
\sqrt{\Theta\,C_1(T)\,T}\cdot
\sqrt{\frac{C_0}{\theta\,g(a)}}.
\label{eq:T2a-bnd}
\end{align}
For $y \in (1,a]$, $z \in (a,n)$,
$a(y,z) = A_0\,\omega_\infty(y)\omega_\infty(z)$
from \eqref{eq:kernel-class}. Applying the Cauchy--Schwarz inequality
and \eqref{eq:tail3} with $m=1$ and $m=a$,
\begin{align}
\left|\int_0^t\!\int_1^a\!\int_a^n
\tilde\psi\,a\,f_nf_n\right|
&\le A_0C_{\psi}
\sqrt{\int_0^t\!\left(\int_1^n
\omega_\infty f_n\,dy\right)^{\!2}\!ds}\,
\sqrt{\int_0^t\!\left(\int_a^n
\omega_\infty f_n\,dz\right)^{\!2}\!ds}
\notag\\
&\le A_0C_{\psi}
\sqrt{\frac{C_0}{\theta\,g(1)}}\cdot
\sqrt{\frac{C_0}{\theta\,g(a)}}.
\label{eq:T2b-bnd}
\end{align}
Combining \eqref{eq:T2a-bnd} and \eqref{eq:T2b-bnd},
\begin{align}
\left|\mathcal{T}_2^n(a)\right|
\le A\,C_{\psi}\sqrt{\frac{C_0}{\theta\,g(a)}} \Bigg(\sqrt{\Theta\,C_1(T)\,T}
+\sqrt{\frac{C_0}{\theta\,g(1)}}\Bigg),
\label{eq:T2-bnd}
\end{align}
uniformly in $n \ge a$.

For $y,z \in (a,n)$,
$a(y,z) = A_0\,\omega_\infty(y)\omega_\infty(z)$
from \eqref{eq:kernel-class}. By \eqref{eq:tail3} with $m=a$,
\begin{align}
\left|\mathcal{T}_3^n(a)\right|
\le A_0C_{\psi}
\int_0^t\!\left(\int_a^n
\omega_\infty f_n\,dy\right)^{\!2}\!ds
\le \frac{A_0\,C_{\psi}\,C_0}{\theta\,g(a)},
\label{eq:T3-bnd}
\end{align}
uniformly in $n \ge a$.

Applying the Cauchy--Schwarz inequality to $f$ and using
\eqref{eq:tail3-lim} and \eqref{eq:tail-lim} in place of the
truncated estimates,
\begin{align}
|\mathcal{T}_2(a)| + |\mathcal{T}_3(a)|
\le A\,C_{\psi}\,\sqrt{\frac{C_0}{\theta\,g(a)}} \Bigg(
\sqrt{\Theta\,C_1(T)\,T}
+\sqrt{\frac{C_0}{\theta\,g(1)}}\Bigg)
+  \frac{A_0\, C_{\psi}\,C_0}{\theta\,g(a)},
\label{eq:T23-lim-bnd}
\end{align}
uniformly in $t \in [0,T]$.
Collecting all estimates, for any $a > 1$ and $n \ge a$,
\begin{align*}
&\left|
\int_0^t\!\int_0^n\!\int_0^n
\tilde\psi\,a\,f_nf_n
-\int_0^\infty\!\int_0^\infty
\tilde\psi\,a\,ff
\right| \\
&\le
|\mathcal{T}_1^n(a)-\mathcal{T}_1(a)|
+2|\mathcal{T}_2^n(a)|+|\mathcal{T}_3^n(a)|
+2|\mathcal{T}_2(a)|+|\mathcal{T}_3(a)|.
\end{align*}
Given $\varepsilon > 0$, by \eqref{eq:T2-bnd}, \eqref{eq:T3-bnd},
and \eqref{eq:T23-lim-bnd}, first choose $a$ large enough so that
\begin{align*}
2\,C_{\psi}\,A\,\sqrt{\frac{C_0}{\theta\,g(a)}} \Bigg(
\sqrt{\Theta\,C_1(T)\,T}
+\sqrt{\frac{C_0}{\theta\,g(1)}}\Bigg) + \frac{A_0\,C_{\psi}\,C_0}{\theta\,g(a)}< \frac{\varepsilon}{4},
\end{align*}
which is possible since $C_0/g(a) \to 0$ as $a \to \infty$. Then, for
this fixed $a$, choose $n$ large enough via \eqref{eq:T1-conv} so that
$|\mathcal{T}_1^n(a)-\mathcal{T}_1(a)| < \varepsilon/2$. Hence
\begin{align}
\lim_{n\to\infty}
\int_0^t\!\int_0^n\!\int_0^n
\tilde\psi(y,z)\,&a(y,z)\,f_n(s,y)f_n(s,z)
\,dy\,dz\,ds \notag\\
&=\int_0^\infty\!\int_0^\infty
\tilde\psi(y,z)\,a(y,z)\,f(s,y)f(s,z)
\,dy\,dz\,ds.
\label{eq:rhs-pass}
\end{align}
Passing to the limit in \eqref{eq:tr-wf-pass} using
\eqref{eq:lhs-pass} and \eqref{eq:rhs-pass}, the limit function $f$
satisfies the weak formulation \eqref{eq:def-wf} for every
$\psi \in L^\infty(0,\infty)$ and $t > 0$. Since $f \ge 0$,
$f \in \mathcal{C}([0,\infty);\Xi_0)$, and $f \in L^\infty((0,\infty);\Xi_g^+)$,
we conclude that $f$ is a global weak solution in the sense of
Definition~\ref{DEF1}.

\begin{remark}
Under condition \eqref{eq:Case1} in place of \eqref{eq:Case2}, the same
argument applies on any fixed interval $[0,T]$, replacing
Lemma~\ref{Lem:ZM}(b) and Lemma~\ref{Lem:UI}(b) by their counterparts
(a) and using Lemma~\ref{Lem:time-eqc} with the initial-data-dependent
constant $C_i(T,f^{\mathrm{in}})$, for $i=1,2,3$. This yields a local-in-time weak
solution on $[0,T]$.
\end{remark}
\end{proof}

%%%%%%%%%%%%%%%%%%%%%%%%%%%%%%%%%%%%
\section{\textbf{Mass Conservation}}\label{sec:mass}
%%%%%%%%%%%%%%%%%%%%%%%%%%%%%%%%%%%%

\begin{prop}[Mass conservation]\label{Prop:MC}
Let $T \in (0,\infty]$, $f^{\mathrm{in}} \in \Xi_0 \cap \Xi_g^+$, and let
$f$ be a weak solution to \eqref{NLBE}--\eqref{IC} on $[0,T)$ in the
sense of Definition~\ref{DEF1}. Assume further that
\begin{align}
s \mapsto
\int_0^\infty\!\int_0^\infty
y\,a(y,z)\,f(s,y)\,f(s,z)\,dy\,dz
\;\in\; L^1(0,t)
\label{eq:MC-cond}
\end{align}
for all $t \in (0,T)$. Then $f$ is mass-conserving on $[0,T)$, i.e.,
\begin{align}
\int_0^\infty x\,f(t,x)\,dx
= \int_0^\infty x\,f^{\mathrm{in}}(x)\,dx
= \Theta,
\qquad t \in [0,T).
\label{eq:MC}
\end{align}
\end{prop}

\begin{proof}
For $R > 0$, set $\psi_R(x) := x\,\mathbf{1}_{(0,R)}(x)$,
which belongs to $L^\infty(0,\infty)$ and is thus an admissible test
function in Definition~\ref{DEF1}. We compute
\begin{align*}
\tilde\psi_R(y,z) = \int_0^y\psi_R(x)\,b(x,y,z)\,dx - \psi_R(y)
\end{align*}
in two cases.

If $y \le R$, then $\psi_R(y) = y$ and $\psi_R(x) = x$ for all
$x \in (0,y)$, so the local mass conservation condition \eqref{LMC} gives
\begin{align}
\tilde\psi_R(y,z)
= \int_0^y x\,b(x,y,z)\,dx - y = 0.
\label{eq:psi-small}
\end{align}
If $y > R$, then $\psi_R(y) = 0$ and
$\psi_R(x) = x\,\mathbf{1}_{(0,R)}(x)$, so
\begin{align}
\tilde\psi_R(y,z)
= \int_0^R x\,b(x,y,z)\,dx
= y - \int_R^y x\,b(x,y,z)\,dx,
\label{eq:psi-large}
\end{align}
where we used \eqref{LMC} in the last step. Since $x \ge 0$ and
$b \ge 0$, both representations in \eqref{eq:psi-large} yield
\begin{align}
0 \le \tilde\psi_R(y,z) \le y,
\qquad y > R.
\label{eq:psi-bnd}
\end{align}
Moreover, for each fixed $(y,z) \in (0,\infty)^2$, taking $R > y$ in
\eqref{eq:psi-small} gives $\tilde\psi_R(y,z) = 0$ for all $R > y$,
so that
\begin{align}
\lim_{R\to\infty}\tilde\psi_R(y,z) = 0.
\label{eq:psi-to-zero}
\end{align}
Combining \eqref{eq:psi-small}--\eqref{eq:psi-bnd},
\begin{align}
0 \le \tilde\psi_R(y,z) \le y,
\qquad (y,z) \in (0,\infty)^2,\quad R > 0.
\label{eq:psi-global-bnd}
\end{align}

Inserting $\psi_R$ into the weak formulation \eqref{eq:def-wf} and
recalling \eqref{eq:psi-small} (so that the integrand vanishes for
$y \le R$), we deduce
\begin{align}
\int_0^R x\,f(t,x)\,dx
- \int_0^R x\,f^{\mathrm{in}}(x)\,dx
= \frac{1}{2}
\int_0^t\!\int_R^\infty\!\int_0^\infty
\tilde\psi_R(y,z)\,a(y,z)\,f(s,y)\,f(s,z)
\,dy\,dz\,ds.
\label{eq:wf-psiR}
\end{align}
Since $\tilde\psi_R \ge 0$, applying the upper bound
\eqref{eq:psi-global-bnd} to the integrand gives
\begin{align}
\int_0^R x\,f(t,x)\,dx
\le \int_0^\infty x\,f^{\mathrm{in}}(x)\,dx
+ \frac{1}{2}
\int_0^t\!\int_0^\infty\!\int_0^\infty
y\,a(y,z)\,f(s,y)\,f(s,z)
\,dy\,dz\,ds
=: \Theta + I(t),
\label{eq:M1-bnd}
\end{align}
where $I(t) < \infty$ for all $t \in (0,T)$ by assumption
\eqref{eq:MC-cond}, and the right-hand side is independent of $R$.
Since $f \ge 0$, letting $R \to \infty$ in \eqref{eq:M1-bnd}
by Fatou's lemma gives
\begin{align}
M_1(f(t))
:= \int_0^\infty x\,f(t,x)\,dx
\le \Theta + I(t) < \infty,
\qquad t \in [0,T).
\label{eq:M1-finite}
\end{align}

Returning to \eqref{eq:wf-psiR} and letting $R \to \infty$, the
left-hand side satisfies
$\int_0^R x\,f(t,x)\,dx \to M_1(f(t))$ and
$\int_0^R x\,f^{\mathrm{in}}(x)\,dx \to \Theta$,
both by monotone convergence (using $f,f^{\mathrm{in}} \ge 0$) and the
finiteness established in \eqref{eq:M1-finite}. For the right-hand side
of \eqref{eq:wf-psiR}, note that $\tilde\psi_R(y,z) \to 0$ pointwise
as $R \to \infty$ by \eqref{eq:psi-to-zero}, and the integrand is
dominated by $y\,a(y,z)\,f(s,y)\,f(s,z)$, which belongs to
$L^1((0,t)\times(0,\infty)^2)$ by \eqref{eq:MC-cond}.
Lebesgue's dominated convergence theorem therefore gives
\begin{align}
\frac{1}{2}
\int_0^t\!\int_R^\infty\!\int_0^\infty
\tilde\psi_R(y,z)\,a(y,z)\,f f\,dy\,dz\,ds
\xrightarrow{R\to\infty} 0.
\label{eq:rhs-to-zero}
\end{align}
Passing to the limit $R \to \infty$ in \eqref{eq:wf-psiR} using
\eqref{eq:M1-finite} and \eqref{eq:rhs-to-zero} yields
$M_1(f(t)) = \Theta$ for all $t \in [0,T)$, which is \eqref{eq:MC}.
\end{proof}

%%%%%%%%%%%%%%%%%%%%%%%%%%%%%%%%%%%%%%%%%%%%%%%%%%%%%%%%%%%%%%%%%%%%%%%%%%%%%%%%%%%%%%%%%%%%%%%%%%%%%%%%%%%%%%%%%%%%%%%%%%%%%%%%%%%%%%%%%%%%%%%%%%%%%%%%%%%%%%%%%%%%%%%%%%%%%%%%%%%%%%%%%%%%%%%%%%%%%%%%%%%%%%%%%%%%%%%%%%%%%%%%%%%%%%%%%%%%%%%%%%%%%%%%%%%%%%%%%%%%%%%%%%%%%%%%%%%%

\begin{corollary}\label{cor:MC-auto}
Let $f$ be the weak solution constructed in Theorem~\ref{Thm:GE}.
Then condition \eqref{eq:MC-cond} is automatically satisfied, and
consequently $f$ is mass-conserving on $[0,T)$.
\end{corollary}

\begin{proof}
Using the kernel structure \eqref{eq:kernel-class}, we write
\begin{align*}
\int_0^\infty\!\int_0^\infty
y\,a(y,z)\,f(s,y)\,f(s,z)\,dy\,dz
=
A_0\bigl(I_1(s)+I_2(s)\bigr)\bigl(J_1(s)+J_2(s)\bigr),
\end{align*}
where
\begin{align*}
I_1(s)&:=\int_0^1 y\,\omega_0(y)\,f(s,y)\,dy,
&
I_2(s)&:=\int_1^\infty y\,\omega_\infty(y)\,f(s,y)\,dy,
\end{align*}
and
\begin{align*}
J_1(s)&:=\int_0^1 \omega_0(z)\,f(s,z)\,dz,
&
J_2(s)&:=\int_1^\infty \omega_\infty(z)\,f(s,z)\,dz.
\end{align*}

We first estimate the terms involving the region $(0,1)$.
By \eqref{eq:w0-growth} and the inequality
$y^{1+\ell}\le y$ for $y\in(0,1)$,
\begin{align*}
I_1(s)
&\le
A_1\int_0^1 y\,f(s,y)\,dy
\le
A_1\,M_1(f(s))
\le
A_1\,\Theta.
\end{align*}
Similarly, since $\omega_0(z)\le A_1 z^\ell\le A_1$
for $z\in(0,1)$,
\begin{align*}
J_1(s)
&\le
A_1\int_0^1 f(s,z)\,dz
\le
A_1\,M_0(f(s))
\le
A_1\,C_1(T).
\end{align*}
Hence,
\begin{align*}
A_0\int_0^t I_1(s)\,J_1(s)\,ds
\le
A\,\Theta C_1(T)\,T
<
\infty.
\end{align*}

Next, since $I_1(s)\le A_1\Theta$, Young's inequality gives
\begin{align*}
I_1(s)J_2(s)
\le
\frac{I_1(s)^2}{2}
+
\frac{J_2(s)^2}{2}.
\end{align*}
Using \eqref{eq:tail3-lim} with $m=1$, we obtain
\begin{align*}
\int_0^T J_2(s)^2\,ds
\le
\frac{C_0}{\theta\,g(1)},
\end{align*}
and therefore
\begin{align*}
A_0\int_0^t I_1(s)\,J_2(s)\,ds
\le
A\left(
\frac{\Theta^2\,T}{2}
+
\frac{ C_0}{2\theta g(1)}
\right)
<
\infty.
\end{align*}

To estimate the remaining terms, we use the fact that
$g\in\mathcal{G}_b$, so the function $g(y)/y$
is non-decreasing. Consequently,
\begin{align}
g(y)\ge g(1)\,y,
\qquad y\ge1,
\label{eq:g-lower}
\end{align}
which implies
\begin{align*}
y\,\omega_\infty(y)
\le
\frac{g(y)\,\omega_\infty(y)}{g(1)},
\qquad y\ge1.
\end{align*}

Restricting \eqref{eq:HMBD2-lim} to the region
$(1,\infty)\times(0,1)$ and using \eqref{eq:g-lower},
we get
\begin{align*}
A_0g(1)\int_0^t I_2(s)J_1(s)\,ds
&=
A_0g(1)\int_0^t\!\int_1^\infty\!\int_0^1
y\,\omega_\infty(y)\,\omega_0(z)\,f\,f\,dz\,dy\,ds
\\
&\le
A_0\int_0^t\!\int_1^\infty\!\int_0^1
g(y)\,\omega_\infty(y)\,\omega_0(z)\,f\,f\,dz\,dy\,ds
\\
&\le
\frac{C_0}{\theta},
\end{align*}
and hence
\begin{align*}
\int_0^t I_2(s)J_1(s)\,ds
\le
\frac{C_0}{A_0g(1)\theta}.
\end{align*}

Similarly, restricting \eqref{eq:HMBD2-lim} to $(1,\infty)^2$
yields
\begin{align*}
A_0g(1)\int_0^t I_2(s)J_2(s)\,ds
&\le
A_0\int_0^t\!\int_1^\infty\!\int_1^\infty
g(y)\,\omega_\infty(y)\,\omega_\infty(z)\,f\,f\,dz\,dy\,ds
\\
&\le
\frac{C_0}{\theta},
\end{align*}
so that
\begin{align*}
\int_0^t I_2(s)J_2(s)\,ds
\le
\frac{C_0}{A_0g(1)\theta}.
\end{align*}

Combining the above estimates, we conclude that
\begin{align*}
\int_0^t\!\int_0^\infty\!\int_0^\infty
y\,a(y,z)\,f(s,y)\,f(s,z)\,dy\,dz\,ds
<
\infty,
\qquad t\in(0,T).
\end{align*}
Therefore condition \eqref{eq:MC-cond} holds in $L^1(0,t)$,
and the mass-conservation property follows from
Proposition~\ref{Prop:MC}.
\end{proof}

\subsection*{Acknowledgements}
MA expresses deep gratitude to Jindal Global Business School,
O.P.\ Jindal Global University, for its invaluable support in providing
essential resources.

\end{document}